\documentclass[12pt]{amsart}
\usepackage{amsmath}
\usepackage{geometry}
\usepackage{amsthm}
\usepackage{amsfonts}
\usepackage{comment}
\usepackage{amssymb}
\usepackage{eufrak}
\usepackage{comment}
\usepackage{hyperref}
\usepackage{mathtools}

\numberwithin{equation}{section}
\usepackage{breqn}

\newtheorem{Theorem}{Theorem}[section]

\newtheorem{Lemma}[Theorem]{Lemma}

\theoremstyle{definition}
\newtheorem*{Example}{Example}
\theoremstyle{remark}
\newtheorem*{remark}{Remark}

    \usepackage{tikz}
    \usetikzlibrary{arrows,matrix}

\newcommand{\A}{\mathcal{A}}
\newcommand{\B}{\mathcal{B}}
\newcommand{\C}{\mathcal{C}}
\newcommand{\D}{\mathcal{D}}

\newcommand{\Title}[1]{\begin{center}\Large{Title???}
\normalsize  \end{center}}
\usepackage{cite}

\title{Generalized Andrews-Gordon Identities}

\author[Hannah Larson]{Hannah Larson}
\address{5015 Donald St., Eugene, OR, 97405}
\email{hannahlarson@college.harvard.edu}

\begin{document}

\maketitle

\begin{abstract}
In a recent paper, Griffin, Ono and Warnaar present a framework for Rogers-Ramanujan type identities using Hall-Littlewood polynomials to arrive at expressions of the form 
\[\sum_{\lambda : \lambda_1 \leq m} q^{a|\lambda|}P_{2\lambda}(1,q,q^2,\ldots ; q^{n}) = \text{``Infinite product modular function"}\]
for $a = 1,2$ and any positive integers $m$ and $n$.
A recent paper of Rains and Warnaar presents further Rogers-Ramanujan type identities involving sums of terms $q^{|\lambda|/2}P_{\lambda}(1,q,q^2,\ldots;q^n)$.
It is natural to attempt to reformulate these various identities to match the well-known Andrews-Gordon identities they generalize. Here, we find combinatorial formulas to replace the Hall-Littlewood polynomials and arrive at such expressions.
\end{abstract}

\section{Introduction}
In \cite{GOW}, the authors construct a general framework describing four doubly-infinite families of Rogers-Ramanujan type identities. In this context, the famous Rogers-Ramanujan identities \cite{RR}
\begin{equation}
\sum_{n=0}^\infty \frac{q^{n^2}}{(1-q)\cdots (1-q^n)} = \prod_{n=0}^\infty \frac{1}{(1-q^{5n+1})(1-q^{5n+4})}
\end{equation}
and
\begin{equation}
\sum_{n=0}^\infty \frac{q^{n^2+n}}{(1-q)\cdots (1-q^n)} = \prod_{n=0}^\infty \frac{1}{(1-q^{5n+2})(1-q^{5n+3})}
\end{equation}
are presented as a special case of their Theorem 1.1 through setting the parameters $(m,n) = (1,1)$. Fixing only $n=1$ gives rise to the $i=1$ and $i=m+1$ instances of the well-known Andrews-Gordon identities, \cite{AG}
\begin{equation} \label{AG}
\sum_{r_1 \geq \cdots \geq r_m \geq 0} \frac{q^{r_1^2+\ldots+r_m^2+r_i+\ldots+r_m}}{(q)_{r_1-r_2}\cdots (q)_{r_{m-1}-r_m}(q)_{r_m}} = \frac{(q^{2m+3};q^{2m+3})_{\infty}}{(q)_\infty} \cdot \theta(q^i; q^{2m+3}),
\end{equation}
where we use the standard notation
\[(a)_k = (a;q)_k := \begin{dcases} (1-a)(1-aq) \cdots (1-aq^{k-1}) & \text{if $k \geq 0$} \\
\prod_{j=0}^\infty (1-aq^j) & \text{if $k=\infty$}
 \end{dcases} \]
and
\[ \theta(a;q) := (a;q)_\infty(q/a;q)_\infty.\]
For convenience, we also set
\[ \theta(a_1,\ldots, a_n; q) := \theta(a_1;q)\cdots \theta(a_n;q).\]

The general identities in \cite{GOW} are presented as sums over partitions $\lambda = (\lambda_1, \lambda_2, \ldots )$ of associated Hall-Littlewood polynomials in the form
\begin{equation} \label{even}
\sum_{\lambda : \lambda_1 \leq m} q^{a|\lambda|}P_{2\lambda}(1,q,q^2,\ldots ; q^{n}) = \text{``Infinite product modular function"},
\end{equation}
with $a=1,2$. Meanwhile, the identities in \cite{RW} take the form
\begin{equation}\label{rwform}
\sum_{\lambda : \lambda_1 \leq m} Cq^{|\lambda|/2}P_{\lambda}(1,q,q^2,\ldots ; q^{n}) = \text{``Infinite product modular function"},
\end{equation}
where $C$ is a particular product of Pochhammer symbols.
Given this framework, it is natural to ask to what extent we can reformulate these identities to look like the Andrews-Gordon identities stated above.
Here, we recast the left-hand side of these identities, without reference to partitions or Hall-Littlewood polynomials, to arrive at such explicit expressions. 

The sums appearing on the left-hand sides of identities corresponding to those in \cite{RW} will range over various sets of decreasing integers $s_1^{(j)} \geq \cdots \geq s_m^{(j)}$ related by $s_i^{(j)} \geq s_i^{(j+1)}$ for $0 \leq j \leq n-1$. We will write $s^{(j)}:=s_1^{(j)} +\ldots + s_m^{(j)}$ and use the convention that $s_i^{(n)}=0$ for all $i$ and $s_{m+1}^{(j)} = 0$ for all $j$. For such a collection of integers, we define
\begin{align}  \label{Adef}
\A_{m,n}(s_*) &:= \A_{m,n}(s_*^{(0)}, s_*^{(1)}, \ldots, s_*^{(n)}) \\
 &=
- \frac{n}{2}s^{(0)}+s^{(1)} + \ldots + s^{(n-1)}+\frac{n}{2}\sum_{i=1}^m\sum_{a=1}^n (s_i^{(a-1)}-s_i^{(a)})^2 \notag
\end{align}
and
\begin{align} \label{Bdef}
\B_{m,n}(s_*) &:= \B_{m,n}(s_*^{(0)}, s_*^{(1)},\ldots, s_*^{(n)};q) \\
&= \prod_{i=1}^m \prod_{a=1}^n \frac{(q)_{s_i^{(a-1)}-s_{i+1}^{(a)}}}{(q)_{s_i^{(a-1)}-s_i^{(a)}}(q)_{s_i^{(a)}-s_{i+1}^{(a)}}} \prod_{j=1}^m \frac{1}{(q)_{s_{i}^{(0)}-s_{i+1}^{(0)}}} . \notag
\end{align}
When the identities include only Hall-Littlewood polynomials of even partitions as in \eqref{even}, it will be more convenient to write our sums over integers $r_1 \geq \cdots \geq r_m \geq 0$ and collections $s_1^{(j)} \geq \ldots \geq s_{2m}^{(j)}$ for $1 \leq j \leq n-1$ satisfying $s_i^{(j)}\geq s_i^{(j+1)}$ and $r_{\lceil i/2 \rceil} \geq s_i^{(1)}$. 
In this case, we define
\begin{equation} \label{Cdef}
\C_{m,n}(r_*,s_*) := \C_{m,n}(r_*, s_*^{(1)},\ldots, s_*^{(n)}) = \A_{2m,n}(s_i^{(0)}=r_{\lceil i/2 \rceil}, s_*^{(1)},\ldots, s_*^{(n)})
\end{equation}
and
\begin{equation} \label{Ddef}
\D_{m,n}(r_*,s_*;q) := \D_{m,n}(r_*,s_*^{(1)},\ldots, s_*^{(n)}; q) := \B_{2m,n}(s_i^{(0)}=r_{\lceil i/2 \rceil}, s_*^{(1)}, \ldots, s_*^{(n)}).
 \end{equation}
 
 The following is a reformulation of Theorem 1.1 of \cite{GOW} which more closely resembles the Andrews-Gordon identities as stated in \eqref{AG}.
\begin{Theorem} \label{th1}
For positive integers $m$ and $n$, let $\kappa := 2m+2n+1$ and $n' := 2n-1$. Then we have
\begin{align*}
\sum_{r_*, s_*} \D_{m, n'}(r_*,s_*;q^{n'})q^{\C_{m, n'}(r_*,s_*)+r} &= \frac{(q^\kappa;q^\kappa)_\infty^n}{(q)_\infty^n} \prod_{i=1}^n\theta(q^{i+m};q^\kappa) \prod_{1 \leq i < j \leq n} \theta(q^{j-i};q^{i+j-1};q^\kappa) \\
&= \frac{(q^\kappa;q^\kappa)_\infty^m}{(q)_\infty^m} \prod_{i=1}^m \theta(q^{i+1};q^\kappa) \prod_{1 \leq i < j \leq m}\theta(q^{j-i},q^{i+j+1};q^\kappa).
\end{align*}
and
\begin{align*}
\sum_{r_*,s*} \D_{m,n'}(r_*,s_*;q^{n'})q^{\C_{m,n'}(r_*,s_*)+2r} &= \frac{(q^\kappa;q^\kappa)_\infty^n}{(q)^n_\infty} \prod_{i=1}^n \theta(q^i;q^\kappa) \prod_{1 \leq i < j \leq n} \theta(q^{j-i},q^{i+j};q^\kappa) \\
&= \frac{(q^\kappa;q^\kappa)_\infty^m}{(q)_\infty^m} \prod_{i=1}^m\theta(q^i;q^\kappa) \prod_{1\leq i < j \leq m}\theta(q^{j-i},q^{i+j};q^\kappa),
\end{align*}
where the sums range over sets of indices $r_1 \geq \cdots \geq r_m\geq 0$ and $s_1^{(j)} \geq \cdots \geq s_{2m}^{(j)} \geq 0$ for $1 \leq j \leq n'-1$ satisfying $s_i^{(j)} \geq s_i^{(j+1)}$ and $r_{\lceil i/2 \rceil} \geq s_i^{(1)}$.
\end{Theorem}

\begin{remark}
As promised, the Andrews-Gordon identities are easily recognized through setting $n=1$. Since $n'=1$, there are no $s$ indices in the sum, giving
\[
\C_{m,1}(r_*)= -r + \frac{1}{2}\sum_{i=1}^{2m} r_{\lceil i/2 \rceil}^2  = - (r_1 + \ldots + r_m) + r_1^2 + \ldots + r_m^2 
\]
and
\[
\D_{m,1}(r_*;q) = \prod_{j=1}^m \frac{1}{(q)_{r_i-r_{i+1}}} =\frac{1}{(q)_{r_1-r_2}\cdots (q)_{r_{m-1}-r_m}(q)_{r_m}}.
\]
Plugging this into the two identities in the theorem above results directly in \eqref{AG} with $i = 1$ and $i = m+1$ respectively.
\end{remark}

One can obtain similar reformulations of Theorems 1.2 and 1.3 of \cite{GOW}, which we label respectively here.

\begin{Theorem}
For positive integers $m$ and $n$, let $\kappa := 2m+2n+2$ and $n':=2n$. Then we have
\begin{align*}
&\sum_{r_*, s_*} \D_{m, n'}(r_*,s_*;q^{n'})q^{\C_{m, n'}(r_*,s_*)+r} \\
&\qquad \qquad = \frac{(q^2;q^2)_\infty (q^{\kappa/2};q^{\kappa/2})_\infty (q^\kappa;q^\kappa)_\infty^{n-1}}{(q)_\infty^{n+1}} \prod_{i=1}^n\theta(q^i;q^{\kappa/2}) \prod_{1\leq i < j \leq n}\theta(q^{j-i},q^{i+j};q^\kappa) \\
&\qquad \qquad = \frac{(q^\kappa;q^\kappa)_\infty^m}{(q)_\infty^m} \prod_{i=1}^m \theta(q^{i+1};q^\kappa) \prod_{1\leq i < j \leq m}\theta(q^{j-i},q^{i+j+1};q^\kappa).
\end{align*}
\end{Theorem}

\begin{Theorem} \label{th3}
For positive integers $m$ and $n$ with $n \geq 2$, let $\kappa := 2m+2n$ and $n':=2n$. Then we have
\begin{align*}
\sum_{r_*, s_*} \D_{m, n'}(r_*,s_*;q^{n'})q^{\C_{m, n'}(r_*,s_*)+2r}
&=\frac{(q^\kappa;q^\kappa)_\infty^n}{(q^2;q^2)_\infty (q)_\infty^{n-1}} \prod_{1\leq i < j \leq n} \theta(q^{j-i},q^{i+j-1};q^\kappa) \\
&= \frac{(q^\kappa;q^\kappa)_\infty^m}{(q)_\infty^m} \prod_{i=1}^m \theta(q^i;q^\kappa)\prod_{1\leq i < j \leq m}\theta(q^{j-i},q^{i+j};q^\kappa).
\end{align*}
\end{Theorem}

We present an example of the identities that are obtained from these theorems.
\begin{Example} 
Setting $n=1$ in Theorem 1.2 and writing $s_i$ for $s_i^{(1)}$, we have that
{
\begin{align*}
&\! \! \! \sum_{\substack{r_1 \geq \cdots \geq r_m \geq 0 \\ s_1 \geq \cdots \geq s_{2m} \geq 0 \\ r_{\lceil i/2\rceil} \geq s_i}}
\frac{q^{(r_1-s_1)^2+(r_1-s_2)^2+\ldots+(r_m-s_{2m-1})^2+(r_m-s_{2m})^2 + s_1^2 + \ldots + s_{2m}^2 + s-r}}
{(q^2;q^2)_{s_1-s_2}\cdots (q^2;q^2)_{s_{2m-1}-s_{2m}}(q^2;q^2)_{s_2}
 (q^2;q^2)_{r_1-r_2} \cdots (q^2;q^2)_{r_{m-1}-r_m}} \\
&\qquad \qquad \qquad \qquad \qquad \quad \times \frac{(q^2;q^2)_{r_1-s_3}(q^2;q^2)_{r_2-s_5} \cdots (q^2;q^2)_{r_{m-1}-s_{2m-1}} 
}{(q^2;q^2)_{r_1-s_1}(q^2;q^2)_{r_2-r_3} \cdots (q^2;q^2)_{r_m-s_{2m-1}}} \\
\\
 &\qquad=\frac{(q^2;q^2)_\infty(q^{m+2};q^{m+2})_{\infty}}{(q)_\infty^2}\theta(q;q^{m+2}) \\
&\qquad 
= \frac{(q^{2m+4};q^{2m+4})_\infty^m}{(q)_\infty^m}\prod_{i=1}^m\theta(q^{i+1};q^{2m+4}) \prod_{1\leq i < j \leq m}\theta(q^{j-i},q^{i+j+1};q^{2m+4})
\end{align*}}
for any positive integer $m$.
Specializing to $(m,n)=(2,1)$, one finds
{\small
\begin{align*}
& \! \! \! \sum_{\substack{r_1 \geq r_2 \geq 0 \\ s_1 \geq \cdots \geq s_4 \geq 0 \\ r_{\lceil i/2\rceil} \geq s_i}}
\! \!\! \frac{(q^2;q^2)_{{r_1}-s_{3}}
q^{(r_1-s_1)^2 + (r_1-s_2)^2 + (r_2-s_3)^2 + (r_2-s_4)^2+ s_1^2 + \ldots + s_4^2 + s_1 + \ldots + s_4 - r_1-r_2}
}{(q^2;q^2)_{r_1-s_1}(q^2;q^2)_{r_2-s_3}(q^2;q^2)_{s_1-s_2}(q^2;q^2)_{s_2-s_3}(q^2;q^2)_{s_3-s_4}(q^2;q^2)_{s_4}(q^2;q^2)_{r_1-r_2}} \\
&\qquad \quad = \frac{(q^2;q^2)_\infty (q^{4};q^{4})_\infty}{(q)_\infty^{2}} \theta(q;q^{4})
= \frac{(q^8;q^8)_\infty^2}{(q)_\infty^2} \theta(q^{2};q^8)\theta(q^3;q^8) \theta(q,q^{4};q^8).
\end{align*}}
\end{Example}

We now turn to the identities of the form \eqref{rwform}. The following theorems are reformulations of Theorems 5.10--5.12 of \cite{RW}.

\begin{Theorem}
For positive integers $m$ and $n$ let $\kappa := m + 2n + 1$ and $n':=2n$. Then we have
\begin{align*}
&\sum_{s_*} \B_{m,n'}(s_*;q^{n'})q^{\A_{m,n'}(s_*)+\frac{1}{2}s^{(0)}} \\
&\qquad=\frac{(q^\kappa;q^\kappa)_\infty^{n-1}(q^{\kappa/2};q^{\kappa/2})_\infty}{(q;q)_\infty^{n-1}(q^{1/2};q^{1/2})_\infty}\prod_{i=1}^n\theta(q^i;q^{\kappa/2})\prod_{1\leq i < j \leq n}\theta(q^{j-i},q^{i+j};q^\kappa),
\end{align*}
where the sum ranges over sets of decreasing integers $s_1^{(j)} \geq \cdots \geq s_m^{(j)}$ for $0 \leq j \leq n'-1$ satisfying $s_i^{(j)} \geq s_i^{(j+1)}$.
\end{Theorem}

\begin{Theorem}
For positive integers $m$ and $n$ let $\kappa := m + 2n$ and $n' := 2n-1$. Then we have
\begin{align*}
&\sum_{s_*} \B_{m,n'}(s_*;q^{n'}) q^{\A_{m,n'}(s_*)+\frac{1}{2}s^{(0)}} \\
&\qquad=\frac{(q^\kappa;q^\kappa)_\infty^n}{(q;q)^{n-1}_\infty(q^{1/2};q)_\infty(q^2;q^2)_\infty}\prod_{i=1}^n\theta(q^{i+(m-1)/2};q^\kappa)\prod_{1\leq i < j \leq n}\theta(q^{j-i},q^{i+j-1};q^\kappa)
\end{align*}
and
\begin{align*}
&\sum_{s_*} \B_{m,2n}(s_*;q^{2n})\left(\prod_{i=1}^{m-1}(-q^n;q^n)_{s_{i}^{(0)}-s_{i+1}^{(0)}}\right) q^{\A_{m,2n}(s_*)+\frac{1}{2}s^{(0)}} \\
&\qquad=\frac{(q^\kappa;q^\kappa)_\infty^{n-1}(q^{\kappa/2};q^{\kappa/2})_\infty}{(q;q)_\infty^{n-1}(q^{1/2};q)_\infty^2(q^2;q^2)_\infty} \prod_{i=1}^n \theta(q^{i-1/2};q^{\kappa/2}) \prod_{1\leq i < j \leq n}\theta(q^{j-i},q^{i+j-1};q^\kappa).
\end{align*}
\end{Theorem}

\begin{Theorem} \label{b}
For positive integers $m$ and $n$ let $\kappa := m + 2n-1$ and $n':=2n-1$. Then we have
\begin{align*}
&\sum_{s_*}\B_{m,n'}(s_*;q^{n'})\left(\prod_{i=1}^{m-1}(-q^{n-1/2};q^{n-1/2})_{s_{i}^{(0)}-s_{i+1}^{(0)}}\right)q^{\A_{m,n'}(s_*)+\frac{1}{2}s^{(0)}} \\
&\qquad = \frac{(q^\kappa;q^\kappa)_\infty^n}{(q;q)_\infty^{n-1}(q^{1/2};q^{1/2})_\infty} \prod_{i=1}^n\theta(q^{i+m/2-1/2};q^\kappa) \prod_{1 \leq i < j \leq n}\theta(q^{j-i},q^{i+j-2};q^\kappa).
\end{align*}
\end{Theorem}

\begin{remark}
As the authors mention in \cite{RW}, setting $n=1$ in Theorem \ref{b} gives rise to Bressoud's even modulus identities in \cite{B}. Indeed, this is easily recognizable using the parameters defined in \eqref{Adef} and \eqref{Bdef}. Writing $s_i$ for $s_i^{(0)}$, we have
\[\mathcal{A}_{m,1}(s_*) = -\frac{1}{2}(s_1 + \ldots + s_m) + \frac{1}{2}(s_1^2 + \ldots + s_m^2)\]
and
\[
\mathcal{B}_{m,1}(s_*;q) = \prod_{j=1}^m\frac{1}{(q)_{s_{i}-s_{i+1}}}.
\]
Thus, the left hand side above is
\begin{align*}
\sum_{s_1 \geq \cdots \geq s_m \geq 0} \frac{1}{(q;q)_{s_m}}\prod_{i=1}^{m-1}\frac{(-q^{1/2};q^{1/2})_{s_{i}-s_{i+1}}}{(q;q)_{s_{i}-s_{i+1}}}q^{\frac{1}{2}(s_1^2+\ldots + s_m^2)}.
\end{align*}
Putting $q^2$ for $q$ and using the fact that $\frac{(-q;q)_k}{(q^2;q^2)_k} = \frac{1}{(q;q)_k}$, we obtain
\[\sum_{s_1\geq\cdots\geq s_m \geq 0} \frac{q^{s_1^2+\ldots + s_m^2}}{(q)_{s_1-s_2}\cdots (q)_{s_{m-1}-s_m}(q^2;q^2)_{s_m}} = \frac{q^{2m+2};q^{2m+2}}{(q)_\infty}\theta(q^{m+1};q^{2m+2}).\]
These are the corresponding even moduli identities to the odd-modulus Andrews-Gordon identities.
\end{remark}

For integers $s_1, \ldots, s_m$ we write $\text{alt}(s_*) := s_1 - s_2 +\ldots \pm s_m$. The following is a reformulation of Theorem 5.14 of \cite{RW}.

\begin{Theorem} \label{th7}
For positive integers $m$ and $n$, let $\kappa := 2m+2n$. Then we have
\begin{align*}
&{\sum_{s_*}}' \B_{2m,2n}(s_*)\left(\prod_{i=1}^{2m-1}(q^{2n};q^{4n})_{\left\lceil \frac{s_{i}^{(0)}-s_{i+1}^{(0)}}{2} \right\rceil}\right) q^{\A_{2m,2n}(s_*)+\frac{1}{2}s^{(0)}+\mathrm{alt}(s_*^{(0)})} \\
&\qquad = \frac{(q^\kappa;q^\kappa)_\infty^n(-q^{\kappa/2};q^{\kappa})_\infty}{2(q;q)_\infty^{n}} \prod_{i=1}^n \theta(-q^{i-1},q^{i+\kappa/2-1};q^\kappa) \prod_{1 \leq i <  j \leq n}\theta(q^{j-i},q^{i+j-2};q^{\kappa}),
\end{align*}
where the prime on the sum denotes the restriction ``$s_i^{(0)}-s_{i+1}^{(0)}$ is even for $i=1,3,\ldots,2m-1$."
\end{Theorem}

This paper is organized as follows.
In the next section we define the Hall-Littlewood polynomials and recall a key formula from \cite{K,WZ}. This allows us to prove two lemmas re-expressing the Hall-Littlewood polynomials that appear in \eqref{even} and \eqref{rwform}. We then apply these to prove Theorems \ref{th1}--\ref{th7} in the following section.

\section*{Acknowledgements}
This project was carried out during the 2015 REU at Emory University. The author would like to thank Ken Ono for suggesting this problem and Ole Warnaar for pointing out that these methods could also be applied to the identities in \cite{RW}. The author also thanks Michael Mertens for useful discussions and the NSF for its financial support.

\section{Hall-Littlewood $q$-series}
Our proofs of Theorems \ref{th1}--\ref{th7} rely on explicit combinatorial formulas for the Hall-Lilttlewood polynomials appearing on the left-hand side of the identities in \cite{GOW} and \cite{RW}. After defining these objects, we state and prove these two formulas as Lemmas \ref{key} and \ref{evenlm}.

A \textit{partition} $\lambda = (\lambda_1, \lambda_2, \ldots)$ is a decreasing sequence of non-negative integers $\lambda_1 \geq \lambda_2 \geq \cdots$ with a finite number $l(\lambda)$ of nonzero terms. By $2\lambda$ we mean the partition $(2\lambda_1, 2\lambda_2, \ldots)$. To each partition, one can associate a Ferrers-Young diagram whose $i$th row consists of $\lambda_i$ boxes. The \textit{conjugate partition} $\lambda'$ is defined to be the partition associated to the transpose of the Ferrers-Young diagram of $\lambda$. The multiplicity $m_i = m_i(\lambda)$ of an integer $i$ is the number of times it appears in the partition and is equal to $\lambda_{i+1}'-\lambda_i'$.
Given two partitions $\lambda, \mu$ we write $\lambda \subseteq \mu$ if the Ferrers-Young diagram for $\lambda$ is contained in that of $\mu$, in other words if $\lambda_i \leq \mu_i$ for all $i$. Given a partition $\lambda$, with $\lambda_1 \leq n$, the associated \textit{Hall-Littlewood polynomial} is defined as
\[P_\lambda(x_1,\ldots, x_n; q) = \prod_{i=0}^n \frac{(1-q)^{m_i}}{(q)_{m_i}}
\sum_{w \in \mathfrak{S}_n}w\left(x_1^{\lambda_1} \cdots x_n^{\lambda_n} \prod_{i<j} \frac{x_i - qx_j}{x_i-x_j} \right),
\]
where $m_0 := n - l(\lambda)$ and the symmetric group $\mathfrak{S}_n$ acts by permuting the $x_i$. One can extend this definition to symmetric functions in countably many variables as follows. If $p_r = x_1^r + x_2^2 + \ldots$ is the $r$-th power sum and $p_\lambda = \prod_{i\geq 1}p_{\lambda_i}$, then the set $\{p_{\lambda}(x_1,\ldots,x_n)\}_{l(\lambda)\leq n}$ is a $\mathbb{Q}$-basis for the ring of symmetric functions in the $x_i$. Let $\phi_q$ be the ring homomorphism determined by $\phi_q(p_r) = p_r/(1-q^r)$. Then we define the $\textit{modified Hall-Littlewood polynomials}$ by
$P'_\lambda := \phi_q(P_\lambda)$ and
\begin{equation} \label{QtoP}
 Q_\lambda'(x_1, \ldots, x_n;q) := P'_\lambda(x_1,\ldots,x_n;q) \prod_{i\geq 1}(q)_{\lambda_i'-\lambda_{i+1}'}.
 \end{equation}
From the fact that
\[\phi_{q^n}(p_r(1,q,\ldots, q^{n-1})) = \frac{1-q^{nr}}{1-q^r} \cdot \frac{1}{1-q^{nr}} = p_r(1,q,q^2,\ldots),\]
we see that
\begin{equation} \label{PtoP'}
P_{\lambda}(1,q, q^2, \ldots ; q^n) = P_{\lambda}'(1,q,\ldots, q^{n-1};q^n).
\end{equation}

We recall the following combinatorial formula for the modified Hall-Littlewood polynomials \cite{K,WZ},
\begin{equation}\label{Q'}
Q_\lambda'(x_1, \ldots, x_n;q) = \sum \prod_{i=1}^{\lambda_1} \prod_{a=1}^n x_a^{\mu_i^{(a-1)}-\mu_i^{(a)}}q^{\binom{\mu_i^{(a-1)}-\mu_i^{(a)}}{2}}\frac{(q)_{\mu_i^{(a-1)}-\mu_{i+1}^{(a)}}}{(q)_{\mu_i^{(a-1)}-\mu_i^{(a)}} (q)_{\mu_i^{(a)} - \mu_{i+1}^{(a)}}},
\end{equation}
where the sum is over partitions $0 = \mu^{(n)} \subseteq \cdots \subseteq \mu^{(1)} \subseteq \mu^{(0)} = \lambda'$. We will combine \eqref{QtoP}-\eqref{Q'} to arrive at the following expression for the Hall-Littlewood polynomials appearing in the sum sides of the identities we wish to rewrite.

\begin{Lemma} \label{key}
Given a positive integer $m$ and a partition $\lambda$ with $\lambda_1 \leq m$, let $s_i^{(0)} = \lambda_i'$. Then for any positive integer $n$,
we have
\[P_{\lambda}(1, q, q^2, \ldots;q^n) = \! \! \! \! \! \sum_{\substack{s_{1}^{(j)} \geq \cdots \geq s_{m}^{(j)} \\s_i^{(j)} \geq s_i^{(j+1)}}} \! \! \! \! \! \B_{m,n}(s_*;q^n)q^{\A_{m,n}(s_*)},\]
where the sum ranges over decreasing sets of integers $s_i^{(j)}$ with $1 \leq j \leq n-1$ and $\A_{m,n}(s_*)$ and $\B_{m,n}(s_*;q)$ are defined in \eqref{Adef} and \eqref{Bdef}.
\end{Lemma}
\begin{proof}
For convenience, let $s_i^{(n)} = 0$ and $s_{m+1}^{(j)}= 0$ for all $i$ and $j$. The conditions on the indices $s_i^{(j)}$ are equivalent to the condition that the partitions $\mu^{(j)}$ defined by $\mu^{(j)} = (s_1^{(j)}, \ldots, s_{m}^{(j)})$ satisfy $0 = \mu^{(n)} \subseteq \cdots \subseteq \mu^{(1)} \subseteq \mu^{(0)} = \lambda'$. Thus, from \eqref{Q'} we have
\begin{align*} Q'_{\lambda}(1,q,\ldots,q^{n-1};q^{n}) &= \sum_{\substack{s_{1}^{(j)} \geq \cdots \geq s_{m}^{(j)} \\ s_i^{(j)} \geq s_i^{(j+1)}}}
 \prod_{i=1}^{m}\prod_{a=1}^nq^{(a-1)(s_i^{(a-1)}-s_i^{(a)})}
q^{n {s_i^{(a-1)}-s_i^{(a)} \choose 2} } \\
&\qquad\qquad \times \frac{(q^n;q^n)_{s_i^{(a-1)}-s_{i+1}^{(a)}}}{(q^n;q^n)_{s_i^{(a-1)}-s_i^{(a)}} (q^n;q^n)_{s_i^{(a)} - s_{i+1}^{(a)}}}.
\end{align*}
Recall that we write $s^{(j)} = s_1^{(j)} +\ldots + s_{m}^{(j)}$.
The power of $q$ appearing in a term of the sum corresponding to an index set $s_i^{(j)}$ is given by
\begin{align*}
&\sum_{i=1}^m\sum_{a=1}^n\left[(a-1)(s_i^{(a-1)}-s_i^{(a)}) + n {s_i^{(a-1)} - s_i^{(a)} \choose 2}\right] \\
&\qquad =\sum_{i=1}^{m} \sum_{a=1}^n\left(a-1-\frac{n}{2}\right)(s_i^{(a-1)}-s_i^{(a)}) + \frac{n}{2}\sum_{i=1}^m\sum_{a=1}^n(s_i^{(a-1)}-s_i^{(a)})^2 \\
&\qquad =\sum_{a=1}^n\left(a-1-\frac{n}{2}\right)(s^{(a-1)}-s^{(a)}) + \frac{n}{2}\sum_{i=1}^m\sum_{a=1}^n(s_i^{(a-1)}-s_i^{(a)})^2 \\
&\qquad =-\frac{n}{2}s^{(0)} + s^{(1)} + \ldots + s^{(n-1)}+ \frac{n}{2}\sum_{i=1}^m\sum_{a=1}^n(s_i^{(a-1)}-s_i^{(a)})^2 \\
&\qquad = \A_{m,n}(s_*).
\end{align*}
In addition, this power of $q$ is multiplied the following product of Pochhammer symbols
\begin{align*}
&\prod_{i=1}^{m} \prod_{a=1}^n\frac{(q^n;q^n)_{s_i^{(a-1)}-s_{i+1}^{(a)}}}{(q^n;q^n)_{s_i^{(a-1)}-s_i^{(a)}} (q^n;q^n)_{s_i^{(a)} - s_{i+1}^{(a)}}}  = \B(s_*; q^n) \prod_{j=1}^m (q^n;q^n)_{s_j^{(0)}-s_{j+1}^{(0)}}.
\end{align*}
Thus, we have
\begin{align*}
Q_{\lambda}'(1,q,\ldots,q^{n-1};q^n)&=\sum_{\substack{s_{1}^{(j)} \geq \cdots \geq s_{m}^{(j)} \\s_i^{(j)} \geq s_i^{(j+1)}}}\prod_{j=1}^m (q^n;q^n)_{s_j^{(0)}-s_{j+1}^{(0)}} \B_{m,n}(s_*;q^{n})q^{\A_{m,n}(s_*)} \\
&=\prod_{j=1}^m (q^n;q^n)_{s_j^{(0)}-s_{j+1}^{(0)}} \sum_{\substack{s_{1}^{(j)} \geq \cdots \geq s_{m}^{(j)} \\s_i^{(j)} \geq s_i^{(j+1)}}} \B_{m,n}(s_*;q^{n})q^{\A_{m,n}(s_*)},
 \end{align*}
so using \eqref{PtoP'} and \eqref{QtoP}, we may write
\begin{align*}
P_{\lambda}(1,q,q^2,\ldots ;q^n) &= P_{\lambda}'(1,q,\ldots,q^{n-1};q^n)\\
&= \frac{Q_{\lambda}'(1,q,\ldots, q^{n-1};q^n)}{\prod_{j=1}^m(q^n;q^n)_{s_j^{(0)}-s_{j+1}^{(0)}}} \\
&= \! \! \! \! \! \sum_{\substack{s_{1}^{(j)} \geq \cdots \geq s_{m}^{(j)} \\s_i^{(j)} \geq s_i^{(j+1)}}} \!\! \! \! \! \B_{m,n}(r_*,s_*;q^{n})q^{\A_{m,n}(r_*,s_*)}. \qedhere
\end{align*}
\end{proof}

We also provide the following formula for Hall-Littlewood polynomials of even partitions.
\begin{Lemma} \label{evenlm}
Given a positive integer $m$ and a partition $\lambda$ with $\lambda_1 \leq m$, let $r_i = \lambda_i'$. Then for any positive integer $n$, we have
\[P_{2\lambda}(1,q,q^2,\ldots;q^n) = \sum_{\substack{s_1^{(j)}\geq \cdots \geq s_{2m}^{(j)} \\ r_{\lceil i/2 \rceil} \geq s_i^{(1)}, s_i^{(j)} \geq s_i^{(j+1)}}} \D_{m,n}(r_*,s_*;q^n) q^{\C_{m,n}(r_*,s_*)},\]
where the sum ranges over sets of decreasing integers $s_i^{(j)}$ for $1 \leq j \leq n-1$ and $\C_{m,n}(r_*,s_*)$ and $\D_{m,n}(r_*,s_*;q)$ are defined in \eqref{Cdef} and \eqref{Ddef}.
\end{Lemma}
\begin{proof}
Applying the previous lemma to the partition $2\lambda$ and recalling the definitions of $\C_{m,n}(r_*,s_*)$ and $\D_{m,n}(r_*,s_*;q)$ results directly in this expression.
\end{proof}

\section{Proof of Theorems}
The proofs of Theorems \ref{th1}--\ref{th3} follow immediately from Lemma \ref{evenlm} and the respectively labeled theorems of \cite{GOW}. A sum over all partitions $\lambda$ with $\lambda_1 \leq m$ is the same as a sum over all partitions whose conjugates have length $l(\lambda') \leq m$. We may represent these partitions by their conjugates, which are specified by indices $r_1 \geq \cdots \geq r_m \geq 0$. This shows in each case that
\begin{align*}
\sum_{r_*,s_*}\D_{m, n'}(r_*,s_*;q^{n'})q^{\C_{m, n'}(r_*,s_*)+ar} &= 
\sum_{\lambda : \lambda_1 \leq m} q^{a|\lambda|}P_{2\lambda}(1,q,q^2,\ldots ; q^{n'}) \\
&= \text{``Infinite product modular function"}
\end{align*}
for $a=1,2$.

The proofs of Theorems 1.4--1.7 follow from Theorems 5.10--5.12, 5.14 of \cite{RW} respectively by rewriting the sums that appear there in a similar way. In this case, we represent the sum over all partitions $\lambda$ with $\lambda_1 \leq m$ as a sum over their conjugates, which we specify by $s_1^{(0)} \geq \cdots \geq s_m^{(0)}$, and use Lemma \ref{key} to rewrite the Hall-Littlewood polynomials. We note that with this notation we have $|\lambda| = s_1^{(0)} + \ldots + s_m^{(0)} = s^{(0)}$
and $m_i(\lambda) = s_{i}^{(0)} - s_{i+1}^{(0)}$. In addition, the number of odd partitions of $\lambda$, written as $\text{odd}(\lambda)$ in Theorem 5.14 of \cite{RW}, is equal to
\[\sum_{i \ \text{odd}} m_i(\lambda) = \sum_{i \ \text{odd}} s_{i}^{(0)} - s_{i+1}^{(0)} = s_1 - s_2 + s_3 - \ldots \pm s_m = \text{alt}(s_*^{(0)}).\]



\begin{thebibliography}{10}

\bibitem{GOW}
M.~Griffin, K.~Ono and S.~O.~Warnaar,
\newblock \textit{A framework of Rogers-Ramanujan identities and their arithmetic properties},
\newblock Duke Mathematical Journal, accepted for publication. \url{arXiv:1401.7718}

\bibitem{RR}
L.~J.~Rogers,
\newblock \textit{Second memoir on the expansion of certain infinite products},
\newblock Proc. London Math. Soc. \textbf{25} (1894), 318--343.

\bibitem{AG}
G.~E.~Andrews,
\newblock \textit{An analytic generalization of the Rogers-Ramanujan identities for odd moduli},
\newblock Prod. Nat. Acad. Sci. USA \textbf{71} (1974), 4082-4085.

\bibitem{RW}
E.~M.~Rains and S.~O.~Warnaar,
\newblock \textit{Bounded Littlewood identities},
\newblock \url{arXiv:1506.02755}

\bibitem{B}
D.~M.~Bressoud,
\newblock \textit{An analytic generalization of the Rogers-Ramanujan identities with interpretation},
\newblock Quart. J. Maths. Oxford (2) \textbf{31} (1980), 385--399.

\bibitem{K}
A.~N.~Kirillov,
\newblock \textit{New combinatorial formula for modified Hall-Littlewood polynomials}, in \textit{q-Series from a Contemporary Perspective}, pp. 283--333, Contmp. Math., 254, AMS, Providence, RI, 2000.

\bibitem{WZ}
S.~O.~Warnaar and W.~Zudilin,
\newblock \textit{Dedekind's $\eta$-function and Rogers-Ramanujan identities},
\newblock Bull. Lond. Math. Soc. \textbf{44} (2012), 1--11.

\end{thebibliography}
\end{document}